\documentclass[11pt]{amsart}
\usepackage{amssymb,amsmath,amsthm,newlfont,enumerate}

\theoremstyle{plain}
\newtheorem{theorem}{Theorem}[section]
\newtheorem{corollary}[theorem]{Corollary}
\newtheorem{lemma}[theorem]{Lemma}
\newtheorem{proposition}[theorem]{Proposition}
\newtheorem{problem}[theorem]{Problem}

\theoremstyle{remark}

\newtheorem*{remarks}{Remarks}

\newcommand{\CC}{\mathbb{C}}

\newcommand{\RR}{\mathbb{R}}

\newcommand{\bx}{\mathbf{x}}
\newcommand{\by}{\mathbf{y}}

\begin{document}

\title[A four-mean theorem]{A four-mean theorem and its application to pseudospectra}

\author{Thomas Ransford}
\address{D\'epartement de math\'ematiques et statistique, Universit\'e Laval, Qu\'ebec (QC) G1V 0A6, Canada}
\email{ransford@mat.ulaval.ca} 

\author{Nathan Walsh}
\address{D\'epartement de math\'ematiques et statistique, Universit\'e Laval, Qu\'ebec (QC) G1V 0A6, Canada}
\email{nathan.walsh.2@ulaval.ca}



\date{9 July 2022}

\begin{abstract}
Let $N$ be an integer with $N\ge 4$.
We show that, if $x_1,\dots,x_N$ and $y_1,\dots,y_N$ are $N$-tuples of strictly positive numbers
whose arithmetic, geometric and harmonic means agree, then 
\[
\max_j x_j <(N-2)\max_j y_j
\quad\text{and}\quad
\min_j x_j <(N-2)\min_j y_j.
\]
A generalized version of this result (where some of $x_j$ and $y_j$ are allowed to be zero) is used to show that, if $N\ge4$ and $A,B$ are $N\times N$ matrices with super-identical pseudospectra, 
then, for every polynomial $p$, we have
\[
\|p(A)\|< \sqrt{N-2}\|p(B)\|,
\]
unless $p(A)=p(B)=0$. This improves a previously known inequality to the point of being sharp, at least for $N=4$.
\end{abstract}



\thanks{Ransford supported by grants from NSERC and the Canada Research Chairs program.
Walsh supported by an NSERC Undergraduate Student Research Award and an FRQNT Supplement.}

\maketitle

\section{Introduction and statement of results}\label{S:intro}

Our main result is the following theorem.

\begin{theorem}\label{T:4mean}
Let $N\ge4$ and let $x_1,\dots,x_N,y_1,\dots,y_N\ge0$.
If 
\begin{align}
\sum_{j=1}^N x_j&=\sum_{j=1}^N y_j,\label{E:am}\\
\prod_{j=1}^N x_j&=\prod_{j=1}^N y_j,\label{E:gm}\\
\intertext{and}
\sum_{k=1}^N\prod_{\substack{j=1\\j\ne k}}^Nx_j&=\sum_{k=1}^N\prod_{\substack{j=1\\j\ne k}}^Ny_j,\label{E:hm}
\end{align}
then
\begin{equation}\label{E:mm}
\max_j x_j\le (N-2)\max_j y_j.
\end{equation}
Equality holds in \eqref{E:mm} if and only if there exists $c\ge0$ such that,
after re\-arranging the $x_j$ and $y_j$ in decreasing order, we have
\[
(x_1,\dots,x_N)=c(N-2,0,\dots,0)
\quad\text{and}\quad
(y_1,\dots,y_N)=c(1,\dots,1,0,0).
\]
\end{theorem}

\begin{corollary}\label{C:4mean}
Let $N\ge4$. If $x_1,\dots,x_N$ and $y_1,\dots,y_N$ are $N$-tuples of strictly positive numbers whose 
arithmetic, geometric and harmonic means agree, then
\begin{equation}\label{E:mmmm}
\max_j x_j<(N-2)\max_j y_j
\quad \text{and}\quad
\min_j x_j<(N-2)\min_j y_j.
\end{equation}
\end{corollary}

\begin{proof}
If all the $x_j$ and $y_j$ are strictly positive, then, under condition~\eqref{E:gm}, the condition \eqref{E:hm}
is equivalent to the statement that $\sum_{j=1}^N1/x_j=\sum_{j=1}^N1/y_j$, in other words, that the harmonic means agree.
Therefore the inequality between maxima in \eqref{E:mmmm} is a consequence of Theorem~\ref{T:4mean}. 

The inequality between minima follows by applying the one for maxima to the $N$-tuples $(1/y_1,\dots,1/y_N)$
and $(1/x_1,\dots,1/x_N)$, whose arithmetic, geometric and harmonic means also agree.
\end{proof}

\begin{remarks}
(i) Corollary~\ref{C:4mean} shows that, 
if the $x_j$ and $y_j$ are constrained so that
three of their means (arithmetic, geometric, harmonic) agree, then the fourth
(maximum or minimum) satisfies the inequality~\eqref{E:mmmm}. 
For this reason, we have dubbed it (and by extension Theorem~\ref{T:4mean})
a `four-mean theorem'. 

(ii) The more complicated formulation of the condition~\eqref{E:hm} allows for the possibility
that some of the $x_j$ and $y_j$ are zero. This is useful since it contains the case of equality.

(iii) Theorem~\ref{T:4mean} is only interesting if $N\ge4$ since, 
if $N=3$, then the conditions \eqref{E:am}, \eqref{E:gm} and \eqref{E:hm}
already imply that $x_j$ are a permutation of the~$y_j$. 
Indeed, they tell us that the polynomials $(t-x_1)(t-x_2)(t-x_3)$
and $(t-y_1)(t-y_2)(t-y_3)$ have the same coefficients, and therefore the same roots.

(iv) There are also two-mean and three-mean theorems. In fact it is obvious that if just
\eqref{E:am} holds, then \eqref{E:mm} is true with $(N-2)$ replaced by $N$, and this is optimal. Also,
if \eqref{E:am} and \eqref{E:gm} both hold, then \eqref{E:mm} is true with $(N-2)$ replaced by $(N-1)$,
and again this is optimal.
However, this is  less obvious, and indeed this three-mean theorem is a step on the route
to establishing Theorem~\ref{T:4mean}.
\end{remarks}

Our motivation for studying constraints and inequalities such as those in Theorem~\ref{T:4mean}
arises from a problem in the theory of pseudospectra of matrices,
which we now describe. For background on pseudospectra, 
we refer to the book \cite{TE05} and the survey article  \cite{Ra10}.

It is known that the pseudospectra of a matrix $A$, namely the level sets of $\|(A-zI)^{-1}\|$,
do not suffice to determine the operator norm of polynomials of $A$, see e.g.\ \cite[Section~47]{TE05}
and \cite[Section~2]{Ra10}, in particular \cite[Theorem~2.3]{Ra10}.
In an attempt to overcome this problem, the authors of \cite{FR09} proposed looking at not only 
the lowest singular values of $A-zI$ for $z\in\CC$ (namely the reciprocals of $\|(A-zI)^{-1}\|$), but all the singular values.
They showed that this information
is sufficient to determine the operator norm of any polynomial of $A$,
up to a factor depending only on the dimension.

Here is a more precise formulation of their result.
We say that two complex $N\times N$ matrices $A,B$  have \emph{super-identical pseudospectra} if 
\[
s_j(A-zI)=s_j(B-zI)\quad(j=1,2,\dots,N,~z\in\CC),
\]
where $s_1,\dots,s_N$ denote the singular values, ordered so that $s_1\ge \cdots\ge s_N$.
It was shown in \cite[Theorem~1.3]{FR09} that, if $A,B$
are $N\times N$ matrices with super-identical pseudospectra and $p$ is a polynomial, then
\begin{equation}\label{E:FRineq}
\|p(A)\|\le\sqrt{N} \|p(B)\|.
\end{equation}

It is not obvious, \emph{a priori}, whether the constant $\sqrt{N}$ can be replaced by~$1$ in this inequality.
In fact it can if $N=1,2,3$ (see \cite[Section~4]{Ra10}), but an example
constructed in \cite{FR09} shows that this is no longer the case when $N\ge 4$.
However, the authors of \cite{FR09} mentioned that they did not know whether the bound $\sqrt{N}$ 
in \eqref{E:FRineq} is optimal.
The following result shows that it is not, and replaces it with a bound that is sharp, at least when $N=4$.

\begin{theorem}\label{T:FRsharp}
Let $N\ge4$ and let $A,B$ be complex $N\times N$ matrices with super-identical pseudospectra. 
Then, for any polynomial $p$, we have
\begin{equation}\label{E:FRsharp}
\|p(A)\|<\sqrt{N-2}\|p(B)\|,
\end{equation}
unless $p(A)=p(B)=0$.
The constant $\sqrt{N-2}$ is sharp at least  if $N=4$.
\end{theorem}

The proof of \eqref{E:FRsharp} is based on the four-mean theorem, Theorem~\ref{T:4mean}.
The proof of sharpness in the case $N=4$ is based on the example from \cite{FR09} mentioned above,
which will be described in detail in Section~\ref{S:sharpness} below.
The same example also permits us to deduce a related result, which we now describe.

Armentia, Gracia and Velasco \cite{AGV12} showed that, 
if $A,B$ have super-identical pseudospectra, 
then they are similar, in other words, there exists an invertible matrix $W$ such that
\begin{equation}\label{E:similar}
B=W^{-1}AW.
\end{equation}
In this case, $p(B)=W^{-1}p(A)W$ for every polynomial $p$,
and so 
\[
\|p(B)\|\le\|W^{-1}\|\|p(A)\|\|W\|.
\]
It is thus tempting to believe that \eqref{E:FRsharp} and \eqref{E:similar}
can be subsumed in a single result in which \eqref{E:similar} holds with
$\|W\|\|W^{-1}\|\le \sqrt {N-2}$. Even if this is false, one might hope that, at the very least,
$W$ may be chosen so that $\|W\|\|W^{-1}\|\le C(N)$ for some constant $C(N)$ depending only on~$N$.
The following theorem shows that, perhaps surprisingly, 
there is no such result. 

\begin{theorem}\label{T:quantsim}
Given $M>0$, there exist $4\times 4$ matrices $A,B$ with super-identical pseudospectra such that
\begin{equation}\label{E:quantsim}
\inf\Bigl\{\|W\|\|W^{-1}\|: B=W^{-1}AW\Bigr\}>M.
\end{equation}
\end{theorem}

The rest of the paper is organized as follows.
In Section~\ref{S:3mean} we establish the three-mean theorem mentioned earlier,
which is then used in Section~\ref{S:4mean} to prove the four-mean theorem, Theorem~\ref{T:4mean}.
Theorem~\ref{T:FRsharp} is deduced in Section~\ref{S:FRsharp},
except for the sharpness statement, which is established in Section~\ref{S:sharpness},
where Theorem~\ref{T:quantsim} is also proved.


\section{Three-mean theorem}\label{S:3mean}

Our goal in this section is to establish the following theorem.

\begin{theorem}\label{T:3mean}
Let $N\ge3$ and let $x_1,\dots,x_N,y_1,\dots,y_N\ge0$.
If 
\[
\sum_{j=1}^N x_j=\sum_{j=1}^N y_j
\quad\text{and}\quad
\prod_{j=1}^N x_j=\prod_{j=1}^N y_j,
\]
then
\begin{equation}\label{E:mm3}
\max_j x_j\le (N-1)\max_j y_j.
\end{equation}
Equality holds in \eqref{E:mm3} if and only if there exists $c\ge0$ such that,
after re\-arranging the $x_j$ and $y_j$ in decreasing order, we have
\[
(x_1,\dots,x_N)=c(N-1,0,\dots,0)
\quad\text{and}\quad
(y_1,\dots,y_N)=c(1,\dots,1,1,0).
\]
\end{theorem}

We shall prove this theorem by reformulating it as an optimization result.
Since the result obviously holds if all the numbers $y_j$ are equal to zero,
we can suppose that at least one of them is non-zero.
Normalizing so that $\max_j x_j =x_1$ and $\max_j y_j=y_1=1$, 
we are led to consider the following problem.

\begin{problem}\label{Pb:3mean}
Let $N\ge3$. 
Maximize $x_1$ subject to the following constraints:
\begin{equation}\label{E:constraints3mean}
\left\{
\begin{aligned}
&x_j\ge0 ~(j=1,\dots,N),\\
&y_1=1,~0\le y_j\le 1 ~(2\le j\le N),\\
&\sum_{j=1}^N x_j=\sum_{j=1}^N y_j,\\
&\prod_{j=1}^N x_j=\prod_{j=1}^N y_j.
\end{aligned}
\right.
\end{equation}
\end{problem}

Here is the solution to Problem~\ref{Pb:3mean}, which establishes Theorem~\ref{T:3mean}.

\begin{theorem}\label{T:soln3mean}
Let $N\ge3$. The maximum value of $x_1$ subject to the constraints \eqref{E:constraints3mean} is $N-1$, 
attained uniquely when
$x_2=\cdots=x_N=0$ and all but one of the $y_j$ are equal to $1$, the remaining one being equal to $0$.
\end{theorem}

\begin{proof}
Let $S$ be the set of $(\bx,\by)\in\RR^N\times \RR^N$ that satisfy the constraints \eqref{E:constraints3mean}.
Clearly the numbers $y_j$ lie in $[0,1]$. Also it is easy to see that the numbers $x_j$ lie in $[0,N]$. Therefore $S$ is a compact set.
The function $(\bx,\by)\mapsto x_1$ is continuous, so it attains its maximum on $S$, say at $(\bx^*,\by^*)$.

If $\bx=(N-1,0,\dots,0)$ and $\by=(1,\dots,1,0)$, then $(\bx,\by)\in S$ and $x_1=N-1$. 
Thus we certainly have $x_1^*\ge N-1$.

We shall show by contradiction that at least one of the terms $x_j^*$ or $y_j^*$ is equal to zero.
Suppose, if possible, that $x_j^*>0$ and $y_j^*>0$  for all $j$.
Applying the standard Lagrange-multiplier argument to 
\[
x_1+\alpha\Bigl(\sum_{j=1}^N x_j-\sum_{j=1}^Ny_j\Bigr)+\beta\Bigl(\sum_{j=1}^N\log x_j-\sum_{j=1}^N\log y_j \Bigr),
\]
we see that
\[
\left\{
\begin{aligned}
1+\alpha+\beta/x^*_1&=0,\\
\alpha+\beta/x^*_j&=0, \quad(2\le j\le N),\\
\alpha+\beta/y^*_j&=0, \quad(1\le j\le N,~\text{if~}y^*_j<1).
\end{aligned}
\right.
\]
From the first equation, $\alpha$ and $\beta$ cannot both be zero,
and from the second, they are in fact both non-zero. Writing $u:=-\beta/\alpha$,
we deduce that
\[
\left\{
\begin{aligned}
x_j^*&=u \quad (2\le j\le N),\\
y_j^*&= u \text{~or~} 1 \quad (1\le j\le N).
\end{aligned}
\right.
\]

There are now two possibilities. The first is that all the $y_j^*$ are equal to $1$.
In this case, the arithmetic and geometric means of the $y^*_j$ are equal to $1$, 
so by \eqref{E:constraints3mean}
the arithmetic and geometric means of the $x^*_j$ are also equal to $1$.
By the case of equality in the AM-GM inequality,
this forces all the $x^*_j$ to be equal to $1$.
This contradicts the fact that $x_1^*\ge N-1$. 

The second possibility is that some $y_j^*=u<1$. This implies that
\[
N-1+(N-1)u\le\sum_{j=1}^N x_j^*=\sum_{j=1}^Ny_j^*\le N-1+u.
\]
Since $N\ge3$ and $u>0$, this is impossible. 

Thus both possibilities lead to contradictions. 
We conclude that at least one of the terms $x_j^*$ or $y_j^*$ is equal to zero, as claimed.

Since $\prod_{j=1}^Nx_j^*=\prod_{j=1}^Ny_j^*=0$, 
at least one $y_j^*=0$, so $\sum_{j=1}^N y^*_j\le N-1$. It follows that
\[
N-1\le x_1^*\le \sum_{j=1}^N x^*_j=\sum_{j=1}^Ny^*_j\le N-1.
\]
Therefore we have equality throughout, which shows that $x_1^*=N-1$ and $x_j^*=0$ for all $j\ge2$,
and also that all but one of the $y_j^*$ satisfy $y_j^*=1$.
This concludes the proof.
\end{proof}


\section{Four-mean theorem}\label{S:4mean}

Following the idea of the preceding section, we shall prove Theorem~\ref{T:4mean}
by formulating it as the solution to an optimization problem. Here is the problem:

\begin{problem}\label{Pb:4mean}
Let $N\ge4$. 
Maximize $x_1$ subject to the following constraints:
\begin{equation}\label{E:constraints4mean}
\left\{
\begin{aligned}
&x_j\ge0 ~(j=1,\dots,N),\\
&y_1=1,~0\le y_j\le 1 ~(2\le j\le N),\\
&\sum_{j=1}^N x_j=\sum_{j=1}^N y_j,\\
&\prod_{j=1}^N x_j=\prod_{j=1}^N y_j,\\
&\sum_{k=1}^N\prod_{\substack{j=1\\j\ne k}}^Nx_j=\sum_{k=1}^N\prod_{\substack{j=1\\j\ne k}}^Ny_j.
\end{aligned}
\right.
\end{equation}
\end{problem}

And here is the solution.

\begin{theorem}\label{T:soln4mean}
Let $N\ge4$. The maximum value of $x_1$ subject to the constraints \eqref{E:constraints4mean} is $N-2$, 
attained uniquely when
$x_2=\cdots=x_N=0$ and all but two of the $y_j$ are equal to $1$, the remaining ones being equal to $0$.
\end{theorem}

Before embarking upon the main proof, it will be convenient to separate out some
algebraic results needed in the course of the argument.

\begin{lemma}\label{L:soln4mean}
Let $N\ge 4$ and let $r\in\{1,2,\dots,N-1\}$.
\begin{enumerate}[\normalfont\rm(i)]
\item If
\[
f(t):=t^{2N-2}-(N-1)t^N+(N-1)t^{N-2}-1, 
\]
then $f(t)\ne0$ for all $t>0$ except $t=1$.
\item If
\[
g(u,v):=(N-1)u^N - (N-r)u^{N-1}-rvu^{N-1}+v^r,
\]
then $g(u,v)\ne0$ in $0<u<v<1$.
\end{enumerate}
\end{lemma}

\begin{proof}
(i) A direct calculation shows that $f(1)=f'(1)=f''(1)=0$, in other words, 
that $f$ has a triple zero at $t=1$.
By Descartes' rule of signs, $f$ has at most three zeros in $(0,\infty)$,
counted according to multiplicity. It follows that $f$ has no zeros in $(0,\infty)$ other than~$1$.

(ii) On the diagonal $u=v$, we have
\[
g(u,u)=(N-1-r)u^{N}-(N-r)u^{N-1}+u^r.
\]
This is identically zero if $r=N-1$. Suppose  that $r<N-1$. Then, by
Descartes' rule of signs again, the right-hand side has at most two zeros in $(0,\infty)$.
On the other hand, a direct verification shows that the right-hand side has a double zero at $u=1$.
Therefore $g(u,u)\ne0$ for all $u\in(0,1)$.
Since $g(u,u)=u^r(1+o(u))$ as $u\to0$, it follows that $g(u,u)>0$ for all $u\in (0,1)$.
Putting together the cases $r=N-1$ and $r<N-1$, we obtain
\begin{equation}\label{E:Fdiag}
g(u,u)\ge0 \quad(0<u<1).
\end{equation}
Now, a simple  computation gives
\begin{equation}\label{E:deriv}
\frac{\partial g}{\partial v}(u,v)=-ru^{N-1}+rv^{r-1}>0
\quad(0<u<v<1).
\end{equation}
Finally,   \eqref{E:Fdiag} and \eqref{E:deriv} together imply that $g(u,v)>0$ in $0<u<v<1$.
\end{proof}

\begin{proof}[Proof of Theorem~\ref{T:soln4mean}]
The proof follows the same general lines as that of Theorem~\ref{T:3mean}, though the details
are a bit more involved.

Let $S$ be the set of $(\bx,\by)\in\RR^N\times \RR^N$ obeying the constraints \eqref{E:constraints4mean}.
Then $S$ is a compact set,
so the function $(\bx,\by)\mapsto x_1$ attains its maximum on $S$, say at $(\bx^*,\by^*)$.

If $\bx=(N-2,0,\dots,0)$ and $\by=(1,\dots,1,0,0)$, then $(\bx,\by)\in S$ and $x_1=N-2$,
so we certainly have $x_1^*\ge N-2$.

We claim that at least one of 
the $x_j^*$ or $y_j^*$ is equal to zero. To prove the claim, we proceed by contradiction. 
So, let us suppose, if possible, that $x_j^*>0$ and $y_j^*>0$ for all~$j$.

In this situation, the final constraint in \eqref{E:constraints4mean}
is equivalent to the condition that $\sum_{j=1}^N1/x_j^*=\sum_{j=1}^N1/y_j^*$.
We may therefore apply  the standard Lagrange-multiplier argument to 
\[
x_1+\alpha\Bigl(\sum_{j=1}^N x_j-\sum_{j=1}^Ny_j\Bigr)+\beta\Bigl(\sum_{j=1}^N\log x_j-\sum_{j=1}^N\log y_j \Bigr)
+\gamma\Bigl(\sum_{j=1}^N \frac{1}{x_j}-\sum_{j=1}^N\frac{1}{y_j}\Bigr)
\]
to obtain
\[
\left\{
\begin{aligned}
1+\alpha+\beta/x^*_1-\gamma/(x_1^*)^2&=0,\\
\alpha+\beta/x^*_j-\gamma/(x_j^*)^2&=0, \quad(2\le j\le N),\\
\alpha+\beta/y^*_j-\gamma/(y_j^*)^2&=0, \quad(1\le j\le N,~\text{if~}y^*_j<1).
\end{aligned}
\right.
\]
The first equality shows that the constants $\alpha,\beta,\gamma$ are not all zero,
and the remaining equalities show that 
\[
\left\{
\begin{aligned}
x_j^*&=u\text{~or~} v  \quad (2\le j\le N),\\
y_j^*&= u\text{~or~} v \text{~or~} 1 \quad (1\le j\le N),
\end{aligned}
\right.
\]
where $u$ and $v$ are the roots of $\alpha t^2+\beta t-\gamma=0$.

If some $x_j^*$ is equal to some $y_k^*$, then the vectors 
 $\bx',\by'\in\RR^{N-1}$ formed by the remaining components  satisfy
the hypotheses of Theorem~\ref{T:3mean} (with $N$ replaced by $N-1$). 
By that theorem, we deduce that $x_1'\le N-2$. 
Since $x_1'=x_1^*\ge N-2$, we actually have equality throughout.
By the case of equality in Theorem~\ref{T:3mean},
all the remaining $x_j'=0$ , which contradicts the supposition that $x_j^*>0$ for all~$j$.
We are thus led to conclude that, in fact,
\[
\left\{
\begin{aligned}
x_j^*&=u  \quad (2\le j\le N),\\
y_j^*&= v \text{~or~} 1 \quad (1\le j\le N).
\end{aligned}
\right.
\]

If all the numbers $y_j^*$  are equal to $1$,
then the arithmetic and geometric means of the $y^*_j$ are equal to $1$, 
so by \eqref{E:constraints4mean}
the arithmetic and geometric means of the $x^*_j$ are also equal to $1$,
which forces all the $x^*_j$ to be equal to $1$.
This contradicts the fact that $x_1^*\ge N-2$. 
We conclude that there exists an integer $r$ with $1\le r\le N-1$
such that exactly $r$ of the $y_j^*$ are equal to $v$ and the remaining $(N-r)$ are equal to~$1$. 
The constraints \eqref{E:constraints4mean} then become
\begin{align}
x_1^*+(N-1)u&=N-r+rv, \label{E:1'}\\
x_1^*u^{N-1}&=v^r, \label{E:2'}\\
\frac{1}{x_1^*}+\frac{N-1}{u}&=N-r+\frac{r}{v}. \label{E:3'}
\end{align}
We also have $u\ne v$ and $v<1$.

The argument now subdivides into two cases, according to whether $u<v$ or $u> v$.

\emph{Case I:}  $u<v$. Eliminating $x_1^*$ from \eqref{E:1'} and \eqref{E:2'} gives
\begin{equation}\label{E:Nr}
v^r+(N-1)u^N=(N-r)u^{N-1}+rvu^{N-1}.
\end{equation}
In other words, $g(u,v)=0$, where $g$ is the function in Lemma~\ref{L:soln4mean}\,(ii).
But by that lemma, $g(u,v)\ne0$ for $0<u<v<1$.
This contradiction concludes the argument for Case~I.

\emph{Case II:} $u> v$.
In this case we must have $r=1$. Indeed, by \eqref{E:1'}, we have
\[
N-r+rv=x_1^*+(N-1)u\ge N-2+(N-1)v,
\]
which, after simplification, leads to
\[
(N-1-r)v\le 2-r.
\]
Since the left-hand side is non-negative, we must have $r\le2$. 
Also, if $r=2$, then the left-hand side is zero, 
which implies that $r=N-1$, contradicting the fact that $N\ge4$. 
The only remaining possibility is that $r=1$, as claimed.

Multiplying together \eqref{E:1'} and \eqref{E:3'}, and recalling that $r=1$, we obtain
\[
1+(N-1)^2+(N-1)\frac{x_1^*}{u}+(N-1)\frac{u}{x_1^*}=(N-1)^2+1+\frac{N-1}{v}+(N-1)v,
\]
which, after simplification, becomes
\[
\frac{x_1^*}{u}+\frac{u}{x_1^*}=v+\frac{1}{v}.
\]
Since the function $t\mapsto (t+1/t)$ is $2$-to-$1$ on $t>0$, 
it follows that either $x_1^*/u=v$ or $x_1^*/u=1/v$.
In the first case, \eqref{E:2'} implies that $vu^N=v$, which in turn implies that $u=1$ and $x_1^*=v<1$, a contradiction. 
So we must have $x_1^*=u/v$. Substituting this information into 
\eqref{E:2'}, we find that $u^N=v^2$ and $x_1^*=u^{1-N/2}$. 
Substituting this into \eqref{E:1'},  and rearranging, we obtain
\[
u^{N-1}-(N-1)u^{N/2}+(N-1)u^{N/2-1}-1=0,
\]
in other words, $f(u^{1/2})=0$, where $f$ is  the polynomial in Lemma~\ref{L:soln4mean}\,(i).
By that lemma,   $f(t)\ne0$ for all $t>0$ except $t=1$.
We conclude that $u=1$, and hence that $x_1^*=1$,  
contradicting the fact that $x_1^*\ge N-2$.
This concludes the argument for Case~II.

Thus, whichever case we are in, we arrive at a contradiction. 
This shows that, as claimed, at least one of the $x_j^*$ or $y_j^*$ is equal to zero.

Because their geometric means are equal,
both vectors $\bx^*,\by^*$ contain a component equal to zero. 
The vectors $\bx',\by'\in\RR^{N-1}$ formed by the remaining components then satisfy
the hypotheses of Theorem~\ref{T:3mean} (with $N$ replaced by $N-1$). 
By that theorem, we deduce that $x_1'\le N-2$. 
Since $x_1'=x_1^*\ge N-2$, we actually have equality throughout.
By the case of equality in Theorem~\ref{T:3mean},
all the remaining $x_j'$ are equal to $0$ and all but one of the $y_j'$ are equal to $1$, the remaining
one being equal to zero. We conclude that $x_1^*=N-2$, 
all the remaining $x_j^*$ are equal to $0$ and all but two of the $y_j^*$ are equal to $1$, the remaining ones 
being equal to zero. This completes the proof of Theorem~\ref{T:soln4mean}, and with it, that of Theorem~\ref{T:4mean}.
\end{proof}

\section{Super-identical pseudospectra}\label{S:FRsharp}

The proof of Theorem~\ref{T:FRsharp} is based on 
Theorem~\ref{T:4mean} and the following lemma.
As before, we write $s_1,\dots,s_N$ to denote the singular values of an $N\times N$ matrix, ordered so that $s_1\ge\dots\ge s_N$.

\begin{lemma}\label{L:FRsharp}
Let $N\ge 4$ and let $A,B$ be complex $N\times N$ matrices with super-identical pseudospectra. Then, for every polynomial $p$,
\begin{align}
\sum_{j=1}^N s_j(p(A))^2&=\sum_{j=1}^N s_j(p(B))^2, \label{E:FRsharp1}\\
\prod_{j=1}^N s_j(p(A))^2&=\prod_{j=1}^N s_j(p(B))^2, \label{E:FRsharp2}\\
\intertext{and}
\sum_{k=1}^N\prod_{\substack{j=1\\j\ne k}}^Ns_j(p(A))^2&=\sum_{k=1}^N\prod_{\substack{j=1\\j\ne k}}^Ns_j(p(B))^2. \label{E:FRsharp3}
\end{align}
\end{lemma}

\begin{proof}
The equality \eqref{E:FRsharp1} was already obtained in the course of the proof of \eqref{E:FRineq}, see \cite[pp. 516--517]{FR09}. We do not repeat the argument here.

The equality \eqref{E:FRsharp2} is a consequence of the result of Armentia, Gracia and Velasco mentioned earlier, according to which
matrices with super-identical pseudospectra are always similar.
Thus $A$ and $B$ are similar, hence also $p(A)$ and $p(B)$.
In particular, $\det(p(A))=\det(p(B))$. Since the absolute value of the determinant is the product of the singular values, we deduce that
\eqref{E:FRsharp2} holds.
 
To establish \eqref{E:FRsharp3}, let us first consider the case when $p(A)$  and $p(B)$ are invertible. Then $p$ has no common zeros with the characteristic polynomial $f$ of $A$, so there exist polynomials $q,r$ such that $pq+fr=1$. Since $f(A)=0$, it follows that
$p(A)q(A)=I$. As $A,B$ are similar, $f(B)=0$ as well, and so also $p(B)q(B)=I$. By \eqref{E:FRsharp1}, with $p$ replaced by $q$, we have
\[
\sum_{j=1}^N s_j(q(A))^2=\sum_{j=1}^N s_j(q(B))^2.
\]
But also we have $q(A)=p(A)^{-1}$, and the singular values of $p(A)^{-1}$ are $1/s_N(p(A)),\dots,1/s_1(p(A))$. 
Likewise for $B$. It follows that
\[
\sum_{j=1}^N \frac{1}{s_j(p(A))^2}=\sum_{j=1}^N \frac{1}{s_j(p(B))^2}.
\]
Multiplying this equation by equation \eqref{E:FRsharp2}, we obtain \eqref{E:FRsharp3}. This proves \eqref{E:FRsharp3} in the case when $p(A)$ and $p(B)$ are invertible. The general case follows by replacing $p(z)$ by $p(z)+\epsilon$, where $\epsilon\ne0$, and then letting $\epsilon\to0$.
\end{proof}

\begin{proof}[Proof of Theorem~\ref{T:FRsharp}]
Let $A,B$ have super-identical pseudospectra, and let $p$ be a polynomial. Then Lemma~\ref{L:FRsharp} implies that the non-negative numbers $x_j:=s_j(p(A))^2$ and $y_j:=s_j(p(B))^2$
satisfy the relations \eqref{E:am}, \eqref{E:gm} and \eqref{E:hm} of Theorem~\ref{T:4mean}. By that theorem, it follows that
\[
\max_js_j(p(A))^2\le (N-2)\max_j s_j(p(B))^2,
\]
in other words, that the operator norms of $p(A)$ and $p(B)$ satisfy
\begin{equation}\label{E:sqineq}
\| p(A)\|^2\le (N-2)\|p(B)\|^2.
\end{equation}

If we have equality in \eqref{E:sqineq}, then,
by the case of equality in Theorem~\ref{T:4mean},
there exists $c\ge0$ such that
\begin{align*}
(s_1(A)^2,\dots,s_N(A)^2)&=c(N-2,0,\dots,0),\\
(s_1(B)^2,\dots,s_N(B)^2)&=c(1,\dots,1,0,0).
\end{align*}
However, since $A,B$ are similar, $p(A)$ and $p(B)$ have the same rank, so they have exactly the same number of non-zero singular values. This can only happen if $c=0$. Thus equality holds in \eqref{E:sqineq} if and 
only if $p(A)=p(B)=0$. 

This completes the proof of Theorem~\ref{T:FRsharp} except for the sharpness statement,
which will be treated in the next section.
\end{proof}


\section{Sharpness results}\label{S:sharpness}

Both the sharpness statement in Theorem~\ref{T:FRsharp}
and the negative result about quantitative similarity, Theorem~\ref{T:quantsim}, 
are consequences of the example contained in the following proposition.

\begin{proposition}\label{P:4x4}
Let $\alpha,\beta\in(0,\pi/4]$, and let
\[
A=
\begin{pmatrix}
0 &\sec\alpha &0 &1\\
0 &0 &\sec\beta\csc\beta &0\\
0 &0 &0 &\csc\alpha\\
0 &0 &0 &0\\
\end{pmatrix}
,~
B=
\begin{pmatrix}
0 &\sec\beta &0 &1\\
0 &0 &\sec\alpha\csc\alpha &0\\
0 &0 &0 &\csc\beta\\
0 &0 &0 &0\\
\end{pmatrix}.
\]
Then $A,B$ have super-identical pseudospectra, and
\begin{equation}\label{E:ratios}
\frac{s_1(A^2)}{s_1(B^2)}=\frac{\cos\alpha}{\cos\beta}
\quad\text{and}\quad
\frac{s_2(A^2)}{s_2(B^2)}=\frac{\sin\alpha}{\sin\beta}.
\end{equation}
\end{proposition}

\begin{proof}
The matrices $A,B$ are taken from \cite[Theorem~5.1]{FR09},
where it is shown that they have super-identical pseudospectra.
It remains to establish \eqref{E:ratios}.

A calculation gives
\[
A^2=
\begin{pmatrix}
0 &0 &\sec\alpha\sec\beta\csc\beta &0\\
0 &0 &0&\csc\alpha\sec\beta\csc\beta\\
0 &0 &0 &0\\
0 &0 &0 &0\\
\end{pmatrix}.
\]
Since $\alpha\in(0,\pi/4]$, we have $\sec\alpha\le\csc\alpha$,
whence
\[
s_1(A^2)=\csc\alpha\sec\beta\csc\beta
\quad\text{and}\quad
s_2(A^2)=\sec\alpha\sec\beta\csc\beta.
\]
Similarly
\[
s_1(B^2)=\csc\beta\sec\alpha\csc\alpha
\quad\text{and}\quad
s_2(B^2)=\sec\beta\sec\alpha\csc\alpha.
\]
The result follows.
\end{proof}

\begin{proof}[Proof of the sharpness statement in Theorem~\ref{T:FRsharp}]
Taking $\beta=\pi/4$ and $\alpha$ close to $0$ in Proposition~\ref{P:4x4},
we see that,  given $\epsilon>0$, there exist $4\times 4$ matrices
$A,B$ with super-identical pseudospectra such that  $\|A^2\|/\|B^2\|>\sqrt{2}-\epsilon$.
This demonstrates that, if $N=4$, then the constant $\sqrt{N-2}$ in \eqref{E:FRsharp} is sharp.
\end{proof}

\begin{proof}[Proof of Theorem~\ref{T:quantsim}]
Given $M>0$, choose $\beta:=\pi/4$ and $\alpha>0$ sufficiently small so that $\sin\beta/\sin\alpha>M$.
Let $A,B$ be the $4\times 4$ matrices with super-identical pseudospectra furnished by Proposition~\ref{P:4x4}.

If $W$ is an invertible $N\times N$ matrix such that $B=W^{-1}AW$,
then $p(B)=W^{-1}p(A)W$ for every polynomial $p$.
It follows that
\[
s_j(p(B))\le \|W^{-1}\|s_j(p(A))\|W\|
\]
for all polynomials $p$ and all $j\in\{1,2,3,4\}$.
In particular, we have
\[
\|W\|\|W^{-1}\|\ge \frac{s_2(B^2)}{s_2(A^2)}=\frac{\sin\beta}{\sin\alpha}>M.
\]
This proves \eqref{E:quantsim} and establishes the result.
\end{proof}


\bibliographystyle{amsplain}
\bibliography{biblist_v6}


\end{document}